\documentclass[a4paper, 12pt]{amsart}

\usepackage{amsthm}
\usepackage{amssymb}
\usepackage{mathrsfs}
\usepackage{latexsym}
\usepackage{stmaryrd}
\usepackage{cases}
\usepackage{amscd}
\usepackage{amsfonts}
\usepackage{txfonts}
\usepackage{amsmath}
\usepackage[all]{xy}
\usepackage{xy}
\usepackage{enumerate}

\usepackage{color}
\definecolor{darkblue}{rgb}{0,0,0.9}
\definecolor{lightblue}{rgb}{.5,.5,.9}
\usepackage[pdftex,breaklinks,colorlinks,filecolor=blue,linkcolor=blue,citecolor=blue,urlcolor=blue,hypertexnames=true,plainpages=false]{hyperref}

\usepackage{tikz}
\usepackage{etoolbox}
\usepackage{comment}
\usepackage{graphics}
\include{definitionsart}

\numberwithin{equation}{section}
\newcommand\C{\mathbb{C}}

\newcommand\Z{\mathbb{Z}}

\theoremstyle{plain}
\newtheorem{theorem}{Theorem}[section]
\newtheorem{lemma}[theorem]{Lemma}
\newtheorem{corollary}[theorem]{Corollary}

\theoremstyle{definition}

\theoremstyle{remark}
\newtheorem{remark}[theorem]{Remark}

\usepackage{color}

\advance \topmargin by -\headheight
\advance \topmargin by -\headsep
     
\evensidemargin \oddsidemargin
\makeatletter
\let\uppercasenonmath\@gobble
\makeatother

\title[Topological classification of complex vector bundles over $8$-dimensional spin$^{c}$ manifolds]{Topological classification of complex vector bundles
over $8$-dimensional spin$^{c}$ manifolds}

\author[Huijun Yang]{Huijun Yang}
\address{School of Mathematics and Statistics, Henan University, Kaifeng 475004, Henan, China}
\email{yhj@amss.ac.cn}

\thanks{}


\date{\today}

\begin{document}

\maketitle

\begin{abstract}
In this paper, complex vector bundles of rank $r$ over $8$-dimensional spin$^{c}$ manifolds are classified in terms of the Chern classes of the complex vector bundles and the cohomology ring of the manifolds, where $r = 3$ or $4$. 
As an application, we got that two rank $3$ complex vector bundles over $4$-dimensional complex projective spaces $\C P^{4}$ are isomorphic if and only if they have the same Chern classes. 
Moreover, the Chern classes of rank $3$ complex vector bundles over $\C P^{4}$ are determined. 
Combing Thomas's and Switzer's results with our work, we can assert that complex vector bundles over $\C P^{4}$ are all classified.
\end{abstract}

\section{Introduction}
\label{s:intro}

In order to classify the holomorphic vector bundles over a fixed complex manifold, 
it is necessary for us to classify the smooth complex vector bundles over it first. 
In this paper, the topological classification of smooth complex vector bundles over manifolds will be investigated.

For a closed oriented smooth $2n$-manifold $M$, 
denote by $\mathrm{Vect}_{\C}^{r}(M)$ the isomorphic classes of rank $r$ smooth complex vector bundles over $M$. 
If $ r = n$, we are in the stable range, and we have a one to one correspondence between $\mathrm{Vect}_{\C}^{n}(M)$ and $\widetilde{K}(M)$ the reduced $K$ group of $M$.
If $ r = 1$, it is known that  the rank $1$ complex vector bundles over $M$ are classified by their first Chern class $c_1$, 
hence we have a bijection between $\mathrm{Vect}_{\C}^{1}(M)$ and $H^{2}(M; \Z)$ for any $n \ge 0$.

In the case $n = 2$. It can be deduced easily from Peterson \cite{pe59} that the rank $2$ complex vector bundles over $M$ are determined by their Chern classes $c_1$ and $c_2$. Moreover,  a classical result of Wu asserts that there is a bijection between $\mathrm{Vect}_{\C}^{2}(M)$ and $H^{2}(M; \Z) \times H^{4}(M; \Z)$.

In the case $n = 3$. 
Firstly, the classification of rank $3$ complex vector bundles over $3$-dimensional complex manifolds are got by B\v{a}nikc\v{a} and Putinar \cite[Theorem 1 (1)]{bp85, bp87}. 
Secondly, the rank $2$ complex vector bundles over $3$-dimensional complex projective space $\C P^{3}$ are classified by  Atiyah and Rees \cite{ar76} , and moreover these bundles over $3$-dimensional complex manifolds are classified by B\v{a}nic\v{a} and Putinar \cite[Theorem 1 (2)]{bp85, bp87}.

In this paper, the classification of complex vector bundles over $8$-dimensional spin$^{c}$ manifolds will be investigated. Since the classification of rank $2$ complex vector bundles have been studied by Switzer, see \cite{sw791} and \cite{sw792} for instance, we will mainly focus on the classification of rank $3$ and $4$ complex vector bundles. Our main results can be stated as follows. 

For a pathwise connected $CW$-complex $X$, d
we will denote by
\begin{align}
\label{eq:bock}
\cdots \rightarrow H^{i}(X; \Z) \xrightarrow{\times 2} H^{i}(X; \Z) \xrightarrow{\rho_2} H^{i}(X; \Z/2) \xrightarrow{\beta} H^{i+1}(X; \Z) \cdots 
\end{align}
the long exact Bockstein sequence associated to the coefficient sequence
\begin{align*}
0 \rightarrow \Z \xrightarrow{\times 2} \Z \rightarrow \Z/2 \rightarrow 0, 
\end{align*}
where $\rho_2$ and $\beta$ are the $\bmod ~2$ reduction and Bockstein homomorphisms respectively,
and denote by 
$\mathrm{Sq}^{2} \colon H^{i}(X; \Z/2) \rightarrow H^{i+2}(X; \Z/2)$
the Steenrod square.

Throughout this paper, $M$ will be an $8$-dimensional closed oriented spin$^{c}$ manifold.
We will fix an element $c \in H^{2}(M; \Z)$ satisfying
\begin{align*}
\rho_{2}(c) = w_{2}(M),
\end{align*}
where $w_{2}(M)$ is the second Stiefel-Whitney class of $M$. We call it a spin$^{c}$ characteristic class of $M$.
Let 
\begin{align*}
\pi_{\ast} \colon \mathrm{Vect}_{\C}^{3}(M) \to  \mathrm{Vect}_{\C}^{4}(M)
\end{align*}
be the map given by $\pi_{\ast} ( \alpha ) = \alpha \oplus \epsilon$, the Whitney sum of $\alpha \in  \mathrm{Vect}_{\C}^{3}(M)$ and $\epsilon$, where $\epsilon$ is the trivial complex vector bundle over $M$ with rank $1$.
Denote by
\begin{align*}
& \mathcal{C}_{4} \colon  \mathrm{Vect}_{\C}^{4}(M) \rightarrow H^{2}(M; \Z) \times H^{4}(M; \Z) \times H^{6}(M; \Z) \times H^{8}(M; \Z) \\
& \mathcal{C}_{3} \colon \mathrm{Vect}_{\C}^{3}(M)  \rightarrow H^{2}(M; \Z) \times H^{4}(M; \Z) \times H^{6}(M; \Z)
\end{align*}
the maps given by 
\begin{align*}
& \mathcal{C}_{4} (\eta) = (c_{1}(\eta), c_{2}(\eta), c_{3}(\eta), c_{4}(\eta)), \quad \text{ for any } ~ \eta \in \mathrm{Vect}_{\C}^{4}(M), \\
& \mathcal{C}_{3} (\eta) = (c_{1}(\eta), c_{2}(\eta), c_{3}(\eta)), \quad \text{ for any } ~ \eta \in \mathrm{Vect}_{\C}^{3}(M),
\end{align*}
respectively, where $c_{i}(\eta)$ is the $i$-th Chern class of $\eta$.
Let $\mathfrak{B}$ and $\mathfrak{T}$ be the quotient groups defined as
\begin{align*}
& \mathfrak{B} : = \frac{ \beta H^{5}(M; \Z/2) }{ \beta \mathrm{Sq}^{2} \rho_{2} H^{3}(M; \Z) }, \\
& \mathfrak{T} : = \frac{ H^{7}(M; \Z) } { \{ f^{\ast}( \gamma_7 ) + u_1 f^{\ast} ( \gamma_5 ) + u_2 f^{\ast} ( \gamma_3 ) + u_{3} f^{\ast} (\gamma_{1}) ~ | ~ f \in [M, U] \} },
\end{align*}
where $u_{i} \in H^{2i}(M; \Z)$, $1 \le i \le 3$, $[M, U]$ is set of homotopy classes of maps of $M$ into the stable unitary group $U$, and where $\gamma_{1}$, $\gamma_{3}$, $\gamma_{5}$ and $\gamma_{7}$ are generators of the exterior algebra
\begin{align*}
H^{\ast}(U; \Z) \cong  \Lambda(\gamma_{1}, \gamma_{3}, \gamma_{5}, \gamma_{7}, \cdots ).
\end{align*}

\begin{theorem}
 \label{thm:main4}
 Let $M$ be an $8$-dimensional spin$^c$ manifold. \vspace{6pt} \\
 $(A)$ For any even dimensional cohomology classes $u_{i} \in H^{2i}(M; \Z), 1 \le i \le 4$,
 \begin{align*}
 (u_1, u_2, u_3, u_4) \in \mathrm{Im} ~ \mathcal{C}_4
 \end{align*}
 if and only if they satisfy the following three conditions
 \vspace{6pt}
\begin{itemize}
\item[(1)] $\mathrm{Sq}^{2} \rho_{2} u_{2} = \rho_{2}( u_3 + u_1 u_2)$,

\vspace{6pt}

\item[(2)] $\langle u_{4}, [M] \rangle  \equiv \langle p_{1}(M) u_2 - u_1^{2} u_2 + u_1 u_3 - u_2^{2}, [M] \rangle  \mod 3$,

\vspace{6pt}

\item[(3)] $\langle u_{4}, [M] \rangle \equiv \langle - u_1^{2} u_2 + u_1 u_3 +   [ 2 u_2^{2} + p_{1}(M) u_2 - 3c^{2} u_2 ] / 4 +  c(u_1 u_2 - u_3) / 2, [M]\rangle \mod2$, 

\vspace{6pt}

\end{itemize}
where $p_1(M)$ is the first Pontrjagin class of $M$,  $[M]$ is the fundamental class of $M$ and 
$\langle ~ \cdot ~ , ~ \cdot ~ \rangle$
is the Kronecker product.

\vspace{6pt}
Moreover, \vspace{6pt} \\
$(B)$ For any $( u_{1}, u_{2}, u_{3}, u_{4} ) \in \mathrm{Im} ~ \mathcal{C}_{4}$, there is a bijection between 
\begin{align*}
\mathcal{C}_{4}^{-1}( u_{1}, u_{2}, u_{3}, u_{4} ) \quad \text{and} \quad \mathfrak{B},
\end{align*}
where $C^{-1}_{4}(u_1, u_2, u_3, u_{4})$ is the pre-image of $(u_1, u_2, u_3, u_{4})$ under the map $C_{4}$.  
It follows that there is a one to one correspondence between
\begin{align*}
 \mathrm{Vect}_{\C}^{4}(M) \quad  \text{and} \quad \mathfrak{B} \times \mathrm{Im} ~ \mathcal{C}_4.
\end{align*}
\end{theorem}

\begin{theorem}
\label{thm:main3}
Let $M$ be an $8$-dimensional spin$^{c}$ manifold. \vspace{6pt} \\
$(C)$ For any $\eta \in \mathrm{Vect}_{\C}^{4}(M)$, the necessary and sufficient condition for $\eta$ to lie in the image of $\pi_{\ast}$ is
\begin{align*}
c_{4}(\eta) = 0.
\end{align*}
Therefore, for any cohomology classes $u_{i} \in H^{2i}(M; \Z)$, $1 \le i \le 3$,  
\begin{align*}
(u_1, u_2, u_3) \in \mathrm{Im} ~ \mathcal{C}_3 \quad \text{if and only if} \quad (u_1, u_2, u_3, 0) \in \mathrm{Im} ~ \mathcal{C}_4.
\end{align*}

Furthermore,  \vspace{6pt} \\
$(D)$ If $(u_1, u_2, u_3) \in \mathrm{Im} ~ \mathcal{C}_3$, then $C^{-1}_{3}(u_1, u_2, u_3)$ is equivalent, as a set, to $\mathfrak{B} \times \mathfrak{T}$.

\end{theorem}

\begin{remark}
Under the condition $(1)$ of Theorem \ref{thm:main4}, one may find from the proof of Theorem \ref{thm:main4} that: \vspace{5pt} \\
$(a)$ the rational number 
\begin{align*}
\langle - u_1^{2} u_2 + u_1 u_3 +   [ 2 u_2^{2} + p_{1}(M) u_2 - 3c^{2} u_2 ] / 4 +  c(u_1 u_2 - u_3) / 2, [M]\rangle
\end{align*}
is an integer, so it make sense to take congruent classes modulo $2$; \vspace{5pt} \\
$(b)$ the congruence in condition $(3)$ in Theorem \ref{thm:main4} is not depend on the choice of $c$.

\end{remark}

\begin{remark}
If we suppose more that $M$ is a $4$-dimensional complex manifold, then we have $p_{1}(M) = c_{1}^{2}(M) - 2 c_{2}(M)$, and $c = c_{1}(M)$. 
Thus one may find that, under the assumptions that $H^{6}(M; \Z)$ and $H^{7}(M; \Z)$ have no $2$-torsion, 
the fact $(A)$ of Theorem \ref{thm:main4}  has been got by 
B\v{a}nic\v{a} and Putinar \cite[Proposition 3.1]{bp06}.
\end{remark}

Peterson \cite[Theorem 3.2]{pe59} tells us that if $H^{6}(M; \Z)$ has no $2$-torsion, then the map $\mathcal{C}_{4}$ is injective.
Obviously, as a corollary of Theorem \ref{thm:main4}, this statement can be generalized as
\begin{corollary}
Let $M$ be a $8$-dimensional spin$^{c}$ manifold. 
Suppose that $\mathfrak{B} = 0$.
Then the map $\mathcal{C}_{4}$ is an injection.
\end{corollary}

At last, as an application, let us consider the classification of complex vector bundles over the $4$-dimensional complex projective space $\C P^{4}$.
Set $t = - c_{1}(\gamma) \in H^{2}(M; \Z)$, where $\gamma$ is the canonical line bundle over $\C P^{4}$.
It is well known that the cohomology ring
$ H^{\ast}(M; \Z) = \Z [t] / \langle t^{5} \rangle $
and the total Chern class $c(\C P^{4}) = (1 + t )^{5}$. 
Hence $p_{1}(\C P^{4}) = 5 t^{2}$, $c = c_{1}(\C P^{4}) = 5 t$, and it can be deduced easily from Theorems \ref{thm:main4} and \ref{thm:main3} that

\begin{corollary}\label{coro:4}
The map 
\begin{align*}
 \mathcal{C}_{4} \colon \mathrm{Vect}_{\C}^{4}(\C P^{4}) \rightarrow H^{2}(\C P^{4}; \Z) \times H^{4}(\C P^{4}; \Z) \times H^{6}(\C P^{4}; \Z) \times H^{8}(\C P^{4}; \Z), 
\end{align*}
given by 
$\mathcal{C}_{4} (\eta) = ( c_{1}(\eta), c_{2}(\eta), c_{3}(\eta), c_{4}(\eta) )$,
is injective. 
\vspace{6pt}

Moreover, 
\begin{align*}
(a_{1} t, a_{2} t^{2}, a_{3} t^{3}, a_{4} t^{4}) \in \mathrm{Im} ~ \mathcal{C}_{4},
\end{align*}
if and only if the integers $a_{i} \in \Z$, $1 \le i \le 4$, satisfy the following two conditions \vspace{6pt}
\begin{itemize}
\item[$(i)$] $2 a_{4} \equiv  a_{2}^{2} +  a_{2} +  a_{1} ( a_{1} a_{2} - a_{3} ) \mod 3$, 
\vspace{6pt}
\item[$(ii)$] $2 a_{4} \equiv a_{2}^{2} + a_{2} + a_{1} a_{2} - a_{3} \mod 4$.
\end{itemize}
\end{corollary}

\begin{corollary} \label{coro:3}
The map 
\begin{align*}
\mathcal{C}_{3} \colon \mathrm{Vect}_{\C}^{3}(\C P^4)  \rightarrow H^{2}(\C P^4; \Z) \times H^{4}(\C P^4; \Z) \times H^{6}(\C P^4; \Z),
\end{align*}
given by 
$\mathcal{C}_{3} (\eta) = ( c_{1}(\eta), c_{2}(\eta), c_{3}(\eta) )$,
is an injection.
\vspace{6pt}

Furthermore, 
\begin{align*}
(a_{1} t, a_{2} t^{2}, a_{3} t^{3}) \in \mathrm{Im} ~ \mathcal{C}_{3},
\end{align*}
if and only if the integers $a_{i} \in \Z$, $1 \le i \le 3$, satisfy the following two conditions \vspace{6pt}
\begin{itemize}
\item[$(i)^{\prime}$] $ a_{2}^{2} +  a_{2} +  a_{1} ( a_{1} a_{2} - a_{3} ) \equiv 0 \mod 3$, 
\vspace{6pt}
\item[$(ii)^{\prime}$] $ a_{2}^{2} + a_{2} + a_{1} a_{2} - a_{3} \equiv 0 \mod 4$.
\end{itemize}

\end{corollary}

\begin{remark}
One may find that the conditions $(i)$, $(ii)$ in Corollary \ref{coro:4} and $(i)^{\prime}$, $(ii)^{\prime}$ in Corollary \ref{coro:3} are just the Schwarzenberger condition (see for instance \cite[Appendix I, p. 165, Theorem 22.4.1]{hi78}). 
\end{remark}

\begin{remark}
In fact, the facts of Corollary \ref{coro:4} have been got by Thomas \cite[Theorem A]{th74}.
Recall that the classification of rank $2$ complex vector bundles over $\C P^4$ have been done by Switzer \cite[Theorem 2]{sw792}.
Therefore, combing their results with Corollary \ref{coro:3}, yields that complex vector bundles over $\C P^{4}$ are all classified.

\end{remark}

After some preliminaries in Section \ref{s:pre}, Theorems \ref{thm:main4} and \ref{thm:main3} will be proved in Section \ref{s:main4} and  \ref{s:main3} respectively.


\section{Preliminaries}
\label{s:pre}

Some preliminaries are needed to prove the main results of this paper.

We work in the category of $CW$-complexes with basepoints. 
For any $CW$-complex $X$, we denote the basepoint of $X$ by $e$,  
the space of loops on X, based at $e$, by $\Omega X$.
We will denote by $[X, Y]$ the set of homotopy classes of maps of $X$ into a $CW$-complex $Y$. 
For a map $f \colon Y \to Z$ between $CW$-complexes $Y$ and $Z$, denote by
\begin{align*}
f_{\ast} \colon [X, Y] \rightarrow [X, Z] \quad \text{and} \quad f^{\ast} \colon [Z, X] \to [Y, X]
\end{align*}
the functions induced by $f$.
Sometimes it is convenient to write $\theta_{\ast}$ and $\theta^{\ast}$ instead of $f_{\ast}$ and $f^{\ast}$ respectively, where $\theta \in [Y, Z]$ denotes the homotopy class of $f$.

Let $K(G, n)$ be the Eilenberg-Mac Lane space of type $(G, n)$ with $G = \Z$ or $\Z/2$. For any $CW$-complex $X$, we will identify 
\begin{align*}
[X, K(G, n)] = H^{n}(X; G)
\end{align*}
in the usual way.
Let
\begin{align*}
\sigma \colon H^{n}(X; G) \to H^{n-1}(\Omega X; G)
\end{align*}
be the cohomology suspension defined as:
for any $h \in H^{n}(X; G)$ represented by the map 
\begin{align*}
h \colon X \to K(G, n),
\end{align*} 
the cohomology suspension of $h$, $\sigma ( h ) \in H^{n-1}(\Omega X; G)$, is the class represented by the map 
\begin{align*}
\Omega h \colon \Omega X \to K(G, n-1).
\end{align*}
Recall from \cite[p. 382, Theorem (3.1)]{wh78} that we have
\begin{lemma}
\label{lem:sigma}
Suppose that $X$ is $m$-connected.
Then the cohomology suspension
\begin{align*}
\sigma \colon H^{n}(X; G) \rightarrow H^{n-1}(\Omega X; G)
\end{align*}
is an isomorphism for $n \le 2m $.
\end{lemma}
Let $B$ be a pathwise connected $CW$-complexes with map
$ \mu \colon B \times B \rightarrow B$,
Suppose that $(B, \mu)$ is an $H$-space.
Let $q \colon E \rightarrow B$ denote the principal fiber space over $B$ with fiber $K(G, n-1)$, and $k$ as classifying map:
\begin{equation*}
\begin{split}
\xymatrix{
          K(G, n-1) ~\ar@{^{(}->}[r]^-{j}        &       E  \ar[d]^{q} &   \\
                 & B \ar[r]^-{k} & K(G, n)        }
\end{split}
\end{equation*}
Then we have an exact sequence 
\begin{align*}
\cdots \rightarrow [X, \Omega B] \xrightarrow{ (\Omega k)_{\ast} }  [X, K(G, n-1)] \xrightarrow{j_{\ast}} [X, E] \xrightarrow{q_{\ast}} [X, B] \xrightarrow{ k_{\ast} } [X, K(G, n)],
\end{align*}
where $k_{\ast}$, $q_{\ast}$ and $j_{\ast}$ are the induced functions, which are generally not homomorphisms, and  where $(\Omega k)_{\ast}$ is the induced homomorphism.

For any $\eta \in [X, B]$, suppose that $n \ge 2$,  let 
\begin{align*}
\Delta(k, \eta) \colon [X, \Omega B] \rightarrow H^{n-1}(X; G),
\end{align*}
be the homomorphism given by (cf. \cite[Corollary 1.4]{jt65})
\begin{align*}
\Delta(k, \eta) (f) = f^{\ast}\sigma(k) + \Sigma (f^{\ast}\sigma(x_i))(\eta^{\ast}(y_i)),
\end{align*}
for any $f \in [X, \Omega B]$, where $x_{i}, y_{i} \in H^{\ast}(B; G)$ are the cohomology classes satisfying
\begin{align*}
\mu^{\ast}( k ) = k \times 1 + 1 \times k + \Sigma ~ x_{i} \times y_{i} \in H^{n}(B \times B; G).
\end{align*}
One may find that the degrees of $x_{i}$ and $y_{i}$ are all greater than or equal to $1$.
Now suppose that 
\begin{align*}
k_{\ast} ( \eta ) = 0 \in [X, K(G, n)]
\end{align*}
is the trivial element. 
It follows from the exact sequence above that $q_{\ast}^{-1} ( \eta ) \neq \emptyset$.
Moreover, it follows from \cite[Theorem 1.2 and Corollary 1.4]{jt65} that
\begin{lemma}[James and Thomas]
\label{lem:delta}
There is a one to one correspondence between 
\begin{align*}
q_{\ast}^{-1}(\eta) \quad  \text{and} \quad \mathrm{Cok} ~ \Delta(k, \eta).
\end{align*}
\end{lemma} 

In particular, as an application, suppose that $\eta = 1 \in [X, B]$ is the trivial element. 
Then we must have $\eta^{\ast} (y_{i}) = 0 \in H^{\ast}(X; G)$ in the definition of $\Delta(k, \eta)$. 
Hence, 
\begin{align*}
\Delta(k, 1) ( f ) = f^{\ast} \sigma (k) = ( \Omega k )_{\ast} ( f )
\end{align*}
by the definitions of $\Delta(k, \eta)$ and the cohomology suspension $\sigma$.
This means that 
\begin{align*}
\Delta(k, 1) = ( \Omega k )_{\ast}.
\end{align*} 
Therefore, it can be deduced easily from Lemma \ref{lem:delta} (or, the exact sequence above) that
\begin{lemma}\label{lem:omegak}
There is a bijection between $\mathrm{Ker} ~ q_{\ast} = q_{\ast}^{-1} ( 1 )$ and $\mathrm{Cok} ~ ( \Omega k )_{\ast}$.
\end{lemma}


\section{ Proof of Theorem \ref{thm:main4} }
\label{s:main4}

In order to prove Theorem \ref{thm:main4}, we need the following notation.

Denote by $BU(n)$ and $BU$ the classifying space of the unitary group $U(n)$ and the stable unitary group $U$ respectively, 
and let 
\begin{align*}
\pi \colon BU(n) \rightarrow BU
\end{align*}
be the inclusion. 
For any $CW$-complex $X$, 
it is known that there is a bijection between
\begin{align*}
\mathrm{Vect}_{\C}^{n}(X) \quad \text{and} \quad  [X, BU(n)] 
\end{align*}
in the natural way.

Recall that $M$ is an $8$-dimensional spin$^{c}$ manifold.
Since
$\pi_{\ast} \colon \pi_{r}(BU(4)) \rightarrow \pi_{r}(BU)$
is isomorphic for $r \le 8$, it follows from \cite[Lemma 4.1]{jt65} that the induced function
\begin{align*}
\pi_{\ast} \colon [M, BU(4)] \rightarrow [M, BU]
\end{align*}
is bijective.
Therefore, to prove Theorem \ref{thm:main4}, it is convenient for us to identify
\begin{align*}
\mathrm{Vect}_{\C}^{4}(M) = [M, BU]
\end{align*}
by the bijections above, and thus we will regard rank $4$ complex vector bundles over $M$ as stable complex vector bundles.

Let us first prove Theorem \ref{thm:main4} $(A)$. 
We need the following lemmas.

Recall that $c \in H^{2}(M; \Z)$ is a fixed spin$^c$ characteristic class of $M$, i.e., 
it satisfies
\begin{align*}
\rho_{2}(c) = w_{2}(M).
\end{align*}

\begin{lemma} 
\label{lem:c4}
For any stable complex vector bundle $\eta$ over $M$, we must have
\begin{align*}
\langle  c_{4}(\eta), [M] \rangle   \equiv & ~ \langle p_{1}(M) c_{2}(\eta) - c_{1}^{2}(\eta) c_{2}(\eta) + c_{1}(\eta) c_{3}(\eta) - c_{2}^{2}(\eta), [M]\rangle \mod 3,  \\
\langle  c_{4}(\eta), [M] \rangle   \equiv & ~ \langle - c_{1}^{2}(\eta) c_{2}(\eta) + c_{1}(\eta) c_{3}(\eta), [M] \rangle  \\
& + \frac{1}{4}  \left\langle   2 c_{2}^{2}(\eta) + p_{1}(M) c_{2}(\eta) - 3 c^{2} c_{2}(\eta) , [M] \right\rangle   \\
& + \frac{1}{2} \langle   c [ c_{1}(\eta) c_{2}(\eta) - c_{3}(\eta) ] , [M] \rangle \mod 2. 
\end{align*}
\end{lemma}

\begin{proof}
Denote by $\hat{\mathfrak{A}}(M)$ the $\mathfrak{A}$-class of $M$ (see for instance \cite[p. 278]{ah59}).
For a complex vector bundle $\eta$ over $M$, let $ch(\eta)$ be the Chern character of $\eta$ and $l_{\eta}$ be the complex line bundle over $M$ with $c_{1}( l_{\eta} ) = c_{1}(\eta)$.
The differential Riemann-Roch Theorem (cf. Atiyah and Hirzebruch \cite[Corollary 1]{ah59}) tells us that the rational number
\begin{align*}
\langle \hat{ \mathfrak{A} }(M) \cdot e^{ \frac{c}{2} } \cdot [ ch( \eta ) - ch(l_{\eta})  - \dim_{\C} \eta + 1 ], [M] \rangle
\end{align*}
is an integer.
Thus,
\begin{align*}
6 ~ \langle \hat{ \mathfrak{A} }(M) \cdot e^{ \frac{c}{2} } \cdot [ ch( \eta ) - ch(l_{\eta})  - \dim_{\C} \eta + 1 ], [M] \rangle \equiv 0 \mod 2, \\
24 ~ \langle \hat{ \mathfrak{A} }(M) \cdot e^{ \frac{c}{2} } \cdot [ ch( \eta ) - ch(l_{\eta})  - \dim_{\C} \eta + 1 ], [M] \rangle \equiv 0 \mod 3.
\end{align*}
Note that we have
\begin{align*}
\hat { \mathfrak{A} } (M) = & ~ 1 - \frac{ p_{1}(M) } { 24 } + \frac{ -4 p_2(M) + 7 p_1^2(M) }{5760}, \\
ch(\eta) = & ~ \dim_{\C} \eta + c_{1}(\eta) + \frac{ c_{1}^{2}(\eta) - 2 c_{2}(\eta) }{ 2 } + \frac{ c_{1}^{3}(\eta) - 3 c_{1}(\eta) c_{2}(\eta) + 3 c_{3}(\eta) }{ 6 } \\
& + \frac{ c_{1}^{4}(\eta) - 4 c_{1}^{2}(\eta) c_{2}(\eta) + 2 c_{2}^{2}(\eta) + 4 c_{1}(\eta) c_{3}(\eta) - 4 c_{4}(\eta) }{ 24 },
\end{align*}
by definition.
Hence, 
\begin{align*}
& ~ \langle \hat{ \mathfrak{A} }(M) \cdot e^{ \frac{c}{2} } \cdot [ ch( \eta ) - ch(l_{\eta})  - \dim_{\C} \eta + 1 ], [M] \rangle \\
= & ~ \frac{1}{6} \langle  - c_{1}^{2}(\eta) c_{2}(\eta) + c_{1}(\eta) c_{3}(\eta) - c_{4}(\eta), [M] \rangle 
     + \frac{1}{4} \langle  c [ c_{3}(\eta) - c_{1}(\eta) c_{2}(\eta) ], [M] \rangle \\
  &  + \frac{1}{24} \langle  2 c_{2}^{2}(\eta) - [ 3 c^{2} - p_{1}(M) ] c_{2}(\eta), [M] \rangle.
\end{align*}
Therefore, the facts of Lemma \ref{lem:c4} are followed by substituting this identity into the congruences above.
\end{proof}

Denote by $M^{\circ} := M - int(D^{8})$ the space obtained from $M$ by removing the interior of a small $8$-disc in $M$.
Let $p \colon M \rightarrow S^{8}$ be the map by collapsing $M ^{\circ}$ to the basepoint,
and  $i \colon M^{\circ} \to M$ be the inclusion map.

\begin{lemma}\label{lem:c46}
For any cohomology class $w \in H^{8}(M; \Z)$ with 
\begin{align*}
\langle w, [M] \rangle \equiv 0 \mod 6,
\end{align*} 
there exists a stable complex vector bundles $\xi^{\prime}$ over $S^{8}$, such that 
$\xi = p^{\ast} ( \xi^{\prime} ) $ is a stable complex vector bundle over $M$ satisfying $c_{1}(\xi) = c_{2}(\xi) = c_{3}(\xi) = 0$, and
\begin{align*}
c_{4}(\xi) = w.
\end{align*}
\end{lemma}

\begin{proof}
Since the degree of $p$ is one, for any $w \in H^{8}(M; \Z)$ with $\langle w, [M] \rangle \equiv 0 \bmod 6$, there exists a class $w^{\prime} \in H^{8}(S^{8}; \Z)$ such that $p^{\ast} ( w^{\prime} ) = w$, and
\begin{align*}
\langle w^{\prime}, [S^{8}] \rangle = \langle w, [M] \rangle \equiv 0 \mod 6.
\end{align*}
Therefore, it follows from Peterson \cite[Theorem 5.1]{pe59} that there must exist a stable complex vector bundle $\xi^{\prime}$ over $S^{8}$ such that $c_4(\xi^{\prime}) = w^{\prime}$.
Moreover, $\xi = p^{\ast} ( \xi^{\prime} )$ is a stable complex vector bundle over $M$ satisfying $c_{1}(\xi) = c_{2}(\xi) = c_{3}(\xi) = 0$, and
$c_{4}(\xi) = w$.
\end{proof}

\begin{proof}[Proof of Theorem \ref{thm:main4} $(A)$]
Suppose that $(u_1, u_2, u_3, u_4) \in \mathrm{Im} ~ \mathcal{C}_4$, 
i.e., there exists $\eta \in \mathrm{Vect}^{4}_{\C}(M)$ such that $u_1 = c_{1}(\eta), u_2 = c_{2}(\eta), u_3 = c_{3}(\eta)$ and $u_4 = c_{4}(\eta)$.
Then the condition $(1)$ in Theorem \ref{thm:main4} is obtained by Wu's explicit formula 
\begin{align*}
\mathrm{Sq}^{2}w_{4} = w_{2} w_{4} + w_{6}
\end{align*}
referring to the universal Stiefel-Whitney classes (see for instance, Milnor and Stasheff \cite[p. 94, Problem 8-A]{ms74}), 
and the conditions $(2)$ and $(3)$ are got by Lemma \ref{lem:c4}.

Conversely, suppose that the cohomology classes $u_{i} \in H^{2i}(M; \Z)$, $1 \le i \le 4$ satisfy the conditions $(1)$-$(3)$ in Theorem \ref{thm:main4}.
Consider the map 
\begin{align*}
\mathcal{C} = (c_{1}, c_{2}, c_{3}) \colon BU \rightarrow K(\Z, 2) \times K(\Z, 4) \times K(\Z, 6) 
\end{align*}
given be the universal Chern classes $c_{1}$, $c_{2}$ and $c_{3}$.
Let $F$ be the homotopy fiber of $\mathcal{C}$. 
It can be deduced easily from the homotopy sequence of this fiber space that 
\begin{align*}\pi_{i}(F)\cong
  \begin{cases}
    0, & i \le 7, i \neq 5; \\
    \mathbb{Z}/2, &  i=5; \\
    \pi_{i}(BU), &  i \ge 8.
  \end{cases}
\end{align*}
Therefore, the Postnikov resolution of the map $\mathcal{C}$ through dimension $9$ can be shown as : 
\begin{equation*}
\begin{split}
\xymatrix{
                 & &       E  \ar[d]^{q}  \ar[r] & K(\Z, 9)  \\
  BU  \ar[urr]^{h} \ar[rr]_-{\mathcal{C}}  & & K \ar[r]^-{k} & K(\mathbb{Z}/2, 6),         }
\end{split}
\end{equation*}
Where 
\begin{align*}
K = K(\Z, 2) \times K(\Z, 4) \times K(\Z, 6)
\end{align*}
and $q$ is a principal fibration with fiber $K(\Z/2, 5)$ and $k$ as classifying map.
We will denote also by 
\begin{align*}
k \in H^{6}(K; \Z/2)
\end{align*}
the class represented by the map $k$.

Since $H^{6}(K; \Z/2)$, 
as a vector space over $\Z/2$, 
is generated by
$\rho_{2} l_{6}$, $\mathrm{Sq}^{2} \rho_{2} l_{4}$,  $\rho_{2} l_{2} \otimes \rho_{2} l_{4}$ and  $\rho_{2} l_{2}^{3}$,
where $l_{n} \in H^{n}( K(\Z, n); \Z )$ is the fundamental class, Wu's explicit formula above implies that
\begin{align*}
k = \mathrm{Sq}^{2} \rho_{2} l_{4} + \rho_{2} l_{6} + \rho_{2} l_{2} \otimes \rho_{2} l_{4}.
\end{align*}
Therefore, the condition $(1)$ in Theorem \ref{thm:main4} 
yields that there exists a complex vector bundle $\eta^{\prime}$ over $M$ such that $c_{1}(\eta^{\prime}) = u_{1}$, $c_{2}(\eta^{\prime}) = u_{2}$ and $c_{3}(\eta^{\prime}) = u_{3}$.
Moreover, it follows from Lemma \ref{lem:c4} that 
\begin{align*}
& \langle c_{4}(\eta^{\prime}), [M] \rangle \equiv \langle p_{1}(M) u_2 - u_{1}^{2} u_{2} + u_{1} u_{3} - u_{2}^{2}, [M]\rangle  \mod 3, \\
& \langle c_{4}(\eta^{\prime}), [M] \rangle \equiv \langle - u_{1}^{2} u_{2} + u_{1} u_{3} + \frac{1}{4} [2 u_{2}^{2} + p_{1}(M) u_{2} - 3 c^{2} u_{2} ] + \frac{1}{2} c ( u_{1} u_{2} - u_{3} ), [M] \rangle \mod 2.
\end{align*}
Thus, combining these congruences with the conditions $(2)$ and $(3)$ in Theorem \ref{thm:main4}, we get that 
\begin{align*}
& \langle u_{4} - c_{4}(\eta^{\prime}), [M] \rangle \equiv 0 \mod 3, \\
& \langle u_{4} - c_{4}(\eta^{\prime}), [M] \rangle \equiv  0 \mod 2.
\end{align*}
That is
\begin{align*}
\langle u_{4} - c_{4}(\eta^{\prime}), [M] \rangle \equiv 0 \mod 6.
\end{align*}
Hence, by Lemma \ref{lem:c4}, there exists a complex vector bundle $\xi$ over $M$ such that 
\begin{align*}
c_{1}(\xi) = c_{2}(\xi) = c_{3}(\xi) = 0 \quad \text{and} \quad c_{4}(\xi) = u_{4} - c_{4}(\eta^{\prime}).
\end{align*}
Let 
\begin{align*}
\eta = \eta^{\prime} + \xi
\end{align*}
be the Whitney sum of $\eta^{\prime}$ and $\xi$. 
Then  $c_{1}(\eta) = u_{1}$,  $c_{2}(\eta) = u_{2}$, $c_{3}(\eta) = u_{3}$, and 
\begin{align*}
c_{4}(\eta) = u_{4}(\eta^{\prime}) + c_{4}(\xi) = u_{4}.
\end{align*}
It follows that $(u_{1}, u_{2}, u_{3}, u_{4}) \in \mathrm{Im} ~ \mathcal{C}_{4}$ and the proof is complete.
\end{proof}

The remainder of this section will be devoted to the proof of Theorem \ref{thm:main4} $(B)$.

We divide the proof into the following lemmas.

For any $CW$-complex $X$, denote by $V_{0}(X)$ the set of stable isomorphic classes of complex vector bundles over $X$ with trivial Chern classes.
Obviously, $V_{0}(X)$ is a subgroup of $[X, BU]$.

\begin{lemma}
For any $(u_1, u_2, u_3, u_4) \in \mathrm{Im} ~ \mathcal{C}_{4}$, there is a bijection between 
\begin{align*}
\mathcal{C}_{4}^{-1} (u_1, u_2, u_3, u_4) \quad  \text{and} \quad \mathrm{V}_{0}(M).
\end{align*}
Therefore, there is a one to one correspondence between $[M, BU]$ (hence $\mathrm{Vect}_{\C}^{4}(M)$) and 
\begin{align*}
\mathrm{V}_{0}(M) \times \mathrm{Im} ~ \mathcal{C}_{4}.
\end{align*}
\end{lemma}

\begin{proof}
For any $u = (u_1, u_2, u_3, u_4) \in \mathrm{Im} ~ \mathcal{C}_{4}$, let $\eta_{u}$ be a fixed stable complex vector bundle over $M$ such that
\begin{align*}
\mathcal{C}_{4} ( \eta_{u} ) = u = (u_1, u_2, u_3, u_4).
\end{align*}
Let
\begin{align*}
\Psi  \colon \mathcal{C}_{4}^{-1} (u_1, u_2, u_3, u_4) \to \mathrm{V}_{0}(M)
\end{align*}
be the map given by $\Psi ( \eta ) = \eta - \eta_{u}$. 
Clearly,  $\Psi$ is bijective.
Note that $\mathrm{V}_{0}(M)$ is not depend on $u = (u_1, u_2, u_3, u_4)$. 
It follows that there is a one to one correspondence between $[M, BU]$ and 
\begin{align*}
\mathrm{V}_{0}(M) \times \mathrm{Im} ~ \mathcal{C}_{4}.
\end{align*}
\end{proof}

\begin{lemma}
The induced homomorphism $i^{\ast} \colon V_{0}(M) \to V_{0}(M^{\circ})$ is bijective.
\end{lemma}

\begin{proof}
Consider the following commutative diagram:
\begin{equation*}
\begin{split}
\xymatrix{
                & V_{0}(M)  \ar@{_{(}->}[d]^{} \ar[r]^{i^{\ast}} &   V_{0}(M^{\circ}) \ar@{_{(}->}[d]^{}      \\
  [S^{8}, BU]  \ar[r]^{p^{\ast}} & [M, BU] \ar[r]^{i^{\ast}} & [M^{\circ}, BU] \ar[r] & [S^{7}, BU],          }
\end{split}
\end{equation*}
where $p^{\ast}$ and $i^{\ast}$ are the induced homomorphisms, and the bottom sequence is exact. 

For any $\eta \in V_{0}(M)$ with $i^{\ast}(\eta) = 0$, there exists $\xi \in [S^{8}, BU]$ such that $p^{\ast}(\xi) = \eta$. 
Since the degree of $p$ is one, 
it follows from $c_{4}(\eta) = 0$ that we must have $c_{4}(\xi) = 0$.
Hence $\xi$ is trivial by Peterson \cite[Theorem 3.2]{pe59}, and so is $\eta$. 
Thus $i^{\ast} \colon V_{0}(M) \to V_{0}(M^{\circ})$ is injective.

Note that we have $[S^{7}, BU] = 0$. 
Therefore, by the exactness of the bottom sequence, the induced homomorphism 
\begin{align*}
i^{\ast} \colon [M, BU] \to [M^{\circ}, BU]
\end{align*}
is surjective.
Hence, 
for any $\eta_{1} \in V_{0}(M^{\circ})$, there exists $\eta_{2} \in [M, BU]$ such that $i^{\ast}(\eta_{2}) = \eta_{1}$. 
Moreover, we have 
\begin{align*}
c_{1}(\eta_{2}) = c_{2}(\eta_{2}) = c_{3}(\eta_{2}) = 0,
\end{align*}
because the induced homomorphisms $i^{\ast} \colon H^{q}(M; \Z) \to H^{q}(M^{\circ}; \Z)$ are isomorphic for $q \le 7$.
Now Theorem \ref{thm:main4} $(A)$
yields that 
\begin{align*}
\langle c_{4}(\eta_{2}), [M] \rangle \equiv 0 \bmod 6.
\end{align*} 
Then there exists $\xi \in [M, BU]$ such that $\xi = p^{\ast} (\xi^{\prime})$ for some $\xi^{\prime} \in [S^{8}, BU]$, and
\begin{align*}
c_{1}(\xi) = c_{2}(\xi) = c_{3}(\xi) = 0,  \quad c_{4}(\xi) = - c_{4}(\eta_{2})
\end{align*} 
by Lemma \ref{lem:c46}.
Set 
\begin{align*}
\eta = \eta_{2} + \xi = \eta_{2} + p^{\ast}( \xi^{\prime} ).
\end{align*} 
It follows that $\eta \in V_{0}(M)$, and 
\begin{align*}
i^{\ast} (\eta) = i^{\ast} ( \eta_{2}) + i^{\ast} p^{\ast} ( \xi^{\prime} ) = \eta_{1},
\end{align*}
which implies that $i^{\ast} \colon V_{0}(M) \to V_{0}(M^{\circ}) $ is surjective, and the proof is complete.
\end{proof}

\begin{lemma}
There is a one to one correspondence between $V_{0}(M^{\circ})$ and $\mathfrak{B} = \frac{\beta H^{5}(M; \Z/2)}{\beta \mathrm{Sq}^{2} \rho_{2} H^{3}(M; \Z)}$.
\end{lemma}

\begin{proof}
Let us first recall from the proof of Theorem \ref{thm:main4} $(A)$ that the postnikov resolution of the fiber space 
\begin{align*}
\mathcal{C} = (c_{1}, c_{2}, c_{3}) \colon BU \rightarrow K(\Z, 2) \times K(\Z, 4) \times K(\Z, 6) 
\end{align*}
through dimension $9$ is as follows:
\begin{equation*}
\begin{split}
\xymatrix{
                & &       E  \ar[d]^{q}  \ar[r] & K(\Z, 9)  \\
  BU  \ar[urr]^{h} \ar[rr]_-{\mathcal{C}} & & K \ar[r]^-{k} & K(\mathbb{Z}/2, 6),          }
\end{split}
\end{equation*}
where $q$ is the principal fibration with fiber $K(\Z/2, 5)$, and $k$ as classifying map.
Therefore, we have the commutative diagram
\begin{equation*}
\begin{split}
\xymatrix{
& & [M^{\circ}, BU] \ar[d]^{h_{\ast}} \ar[dr]^-{\mathcal{C}_{\ast}}& & \\
[M^{\circ}, \Omega K] \ar[r]^-{(\Omega k)_{\ast}} & [M^{\circ}, K(\Z/2, 5)] \ar[r]^-{j_{\ast}} & [M^{\circ}, E] \ar[r]^-{q_{\ast}} & [M^{\circ}, K] \ar[r]^-{k_{\ast}} & [M^{\circ}, K(\Z/2, 6)], 
}
\end{split}
\end{equation*}
where $h_{\ast}$, $q_{\ast}$, $\mathcal{C}_{\ast}$ and $k_{\ast}$ are the induced functions, $j_{\ast}$ is the function induced by the inclusion $j \colon K(\Z/2, 5) \to E$ of the fiber $K(\Z/2, 5)$ into the total space $E$, 
$(\Omega k)_{\ast}$ is an induced homomorphism and where the bottom sequence is exact.

Since the homomorphisms $h_{\ast} \colon \pi_{r}(BU) \rightarrow \pi_{r}(E)$, $r \le 7$,
are isomorphic by the construction of the Postnikov resolution,
note that $M^{\circ}$ can be regard as the $7$-skeleton of $M$, it follows from \cite[Lemma 4.1]{jt65} that 
\begin{align*}
h_{\ast} \colon [M^{\circ}, BU] \to [M^{\circ}, E]
\end{align*}
is injective.  
Therefore, there is a one to one correspondence between 
\begin{align*}
\mathrm{V}_{0}(M^{\circ}) = \mathrm{Ker} ~  \mathcal{C}_{\ast} \quad \text{and} \quad \mathrm{Ker} ~ q_{\ast}.
\end{align*}
Hence a bijection between
\begin{align*}
\mathrm{V}_{0}(M^{\circ}) \quad \text{and} \quad \mathrm{Cok} ~ (\Omega k)_{\ast}
\end{align*}
by Lemma \ref{lem:omegak}.
Therefore, it remains to prove that 
$\mathrm{Cok} ~ (\Omega k)_{\ast}$ is isomorphic to $\mathfrak{B} = \frac{\beta H^{5}(M; \Z/2)}{\beta \mathrm{Sq}^{2} \rho_{2} H^{3}(M; \Z)}$.

Recall that we have
\begin{align*}
k = \mathrm{Sq}^{2} \rho_{2} l_{4} + \rho_{2} l_{6} + \rho_{2} l_{2} \otimes \rho_{2} l_{4} \in H^{6}(K; \Z/2).
\end{align*}
Then it can be deduced easily from the definition of cohomology suspension and Lemma \ref{lem:sigma} that
 \begin{align*}
 \Omega k = \sigma( k ) = \mathrm{Sq}^{2} \rho_{2} l_{3} + \rho_{2} l_{5}.
 \end{align*}
Hence the homomorphism $(\Omega k)_{\ast} \colon [M^{\circ}, \Omega K] \to [M^{\circ}, K(\Z/2, 5)]$ is just the homomorphism
\begin{align*}
(\Omega k)_{\ast} \colon H^{1}(M^{\circ}; \Z) \times H^{3}(M^{\circ}; \Z) \times H^{5}(M^{\circ}; \Z) \to H^{5}(M^{\circ}; \Z/2)
\end{align*}
given by 
\begin{align*}
(\Omega k)_{\ast} (x, y, z) = \mathrm{Sq}^{2} \rho_{2} y + \rho_{2} z
\end{align*}
for any $(x, y, z)$ in $H^{1}(M^{\circ}; \Z) \times H^{3}(M^{\circ}; \Z) \times H^{5}(M^{\circ}; \Z)$.

Note that the homomorphisms
\begin{align*}
H^{r}(M^{\circ}; G) \cong H^{r}(M; G), \quad r \le 7,
\end{align*}
are isomorphic. Then it follows from the Bockstein sequence \eqref{eq:bock} that
\begin{align*}
\mathrm{Cok} ~ (\Omega k)_{\ast} = & \quad \frac{H^{5}(M^{\circ}; \Z/2)}{\rho_{2} H^{5}(M^{\circ}; \Z) + \mathrm{Sq}^{2} \rho_{2} H^{3}(M^{\circ}; \Z)} \\
\cong 
&  \quad \frac{H^{5}(M; \Z/2)}{\rho_{2} H^{5}(M; \Z) + \mathrm{Sq}^{2} \rho_{2} H^{3}(M; \Z)} \\
\cong 
& \quad \frac{H^{5}(M; \Z/2) / \rho_{2} H^{5}(M; \Z) }{ [ \rho_{2} H^{5}(M; \Z) + \mathrm{Sq}^{2} \rho_{2} H^{3}(M; \Z) ] / \rho_{2} H^{5}(M; \Z) } \\
\cong
& \quad \frac{ \beta H^{5}(M; \Z/2)}{ \beta \mathrm{Sq}^{2} \rho_{2} H^{3}(M; \Z) } \\
=
& \quad \mathfrak{B},
\end{align*}
which completes the proof.
\end{proof}


\section{ Proof of Theorem \ref{thm:main3} }
\label{s:main3}

By the identifications in the beginning of Section \ref{s:main4}, we will regard the function
\begin{align*}
\pi_{\ast} \colon \mathrm{Vect}_{\C}^{3} (M) \to \mathrm{Vect}_{\C}^{4} (M).
\end{align*}
defined in section \ref{s:intro} as the fuction
\begin{align*}
\pi_{\ast} \colon [M, BU(3)] \rightarrow [M, BU]
\end{align*}
induced by the inclusion map
$\pi \colon BU(3) \to BU$.

\begin{proof}[Proof of Theorem \ref{thm:main3} $(C)$]
Consider the Postnikov resolution of the inclusion map 
\begin{align*}
\pi \colon BU(3) \to BU.
\end{align*}
Its homotopy fiber is $U/U(3)$.
It is known that we have 
\begin{align*}
& \pi_{r}(U) \cong \pi_{r}(U(5)) \quad \text{for} \quad r \le 9, \\
& \pi_{r}(U(3)) \cong \pi_{r}(U(5)) \quad \text{for} \quad r \le 5,
\end{align*}
and $\pi_{6}(U(5)) = 0$. It follows that 
\begin{align*}
& \pi_{r}(U/U(3)) \cong \pi_{r}(U(5)/U(3)) \quad \text{for} \quad r \le 9, \\
& \pi_{r}(U(5)/U(3)) = 0  \quad \text{for} \quad r \le 6.
\end{align*}
Moreover, Gilmore \cite[pp. 630 - 631]{gi67} tells us that 
\begin{align*}
\pi_{r}(U(5)/U(3))\cong
  \begin{cases}
    \Z, & r = 7, 9; \\
    0, & r = 8.
  \end{cases}
\end{align*}
Therefore, summarize the facts above, we have 
\begin{align*}\pi_{r}(U/U(3))\cong
  \begin{cases}
    0, & r \le 8, r \neq 7; \\
    \mathbb{Z}, &  r = 7; \\
    \Z , &  r = 9.
  \end{cases}
\end{align*}
Then the Postnikov resolution of the map $i$ through dimension $10$ can be shown as
\begin{equation*}
\begin{split}
\xymatrix{
                & &       E  \ar[d]^{q}  \ar[r] & K(\Z, 10)  \\
  BU(3)  \ar[urr]^{h} \ar[rr]_-{\pi} & & BU \ar[r]^-{k} & K(\mathbb{Z}, 8).           }
\end{split}
\end{equation*}
Here $q$ is a principal fibration with fiber $K(\Z, 7)$ and $k$ as the classifying map.
Regard 
\begin{align*}
k \in H^{8}(BU; Z)
\end{align*}
as the class represented by the map $k$.
Since the kernel of the induced homomorphism 
\begin{align*}
\pi^{\ast} \colon H^{8}(BU; \Z) \to H^{8}(BU(3); \Z)
\end{align*}
is generated by the universal Chern class $c_{4}$, 
it follows that $k = c_{4}$.
Thus, recall that the dimension of $M$ is $8$, the proof is complete. 
\end{proof}

\begin{proof}[Proof of Theorem \ref{thm:main3} $(D)$]
Let us first consider the induced function
\begin{align*}
\pi_{\ast} \colon [M, BU(3)] \to [M, BU].
\end{align*}
For any $\eta \in \mathrm{Im} ~ \pi_{\ast}$, we claim that there is a bijection between $\pi_{\ast}^{-1} ( \eta )$ and 
\begin{align*}
\mathfrak{T}  : = \frac{ H^{7}(M; \Z) } { \{ f^{\ast}( \gamma_7 ) + u_1 f^{\ast} ( \gamma_5 ) + u_2 f^{\ast} ( \gamma_3 ) + u_{3} f^{\ast} (\gamma_{1}) ~ | ~ f \in [M, U] \} },
\end{align*}
where $\gamma_{2r - 1} = \sigma ( c_{r} ) \in H^{2r-1}(U)$, $1 \le r \le 4$, are the generators of the exterior algebra
\begin{align*}
H^{\ast}(U; \Z) \cong \Lambda(\gamma_{1}, \gamma_{3}, \gamma_{5}, \gamma_{7}, \cdots).
\end{align*}
For any $( u_{1}, u_{2}, u_{3} ) \in \mathrm{Im} ~ \mathcal{C}_{3} $, it follows from Theorem \ref{thm:main3} $(C)$ that 
\begin{align*}
\mathcal{C}_{4}^{-1} ( u_{1}, u_{2}, u_{3}, 0 ) \subset \mathrm{Im} ~ \pi_{\ast},
\end{align*}
and from Theorem \ref{thm:main4} $(B)$ that there is a bijection between $\mathcal{C}_{4}^{-1} ( u_{1}, u_{2}, u_{3}, 0 )$ and $\mathfrak{B}$.
Moreover, we must have
\begin{align*}
\mathcal{C}_{3}^{-1} ( u_{1}, u_{2}, u_{3} ) =  \pi_{\ast}^{-1} ( \mathcal{C}_{4}^{-1} ( u_{1}, u_{2}, u_{3}, 0 ) )
\end{align*}
by the definitions of $\mathcal{C}_{3}$, $\mathcal{C}_{4}$ and $\pi_{\ast}$.
Therefore, note that $\mathfrak{T}$ is only depend on the Chern classes of $\eta$, the claim above implies that there is a bijection between
\begin{align*}
\mathcal{C}_{3}^{-1} ( u_{1}, u_{2}, u_{3} ) \quad \text{and} \quad \mathfrak{B} \times \mathfrak{T},
\end{align*}
which completes the proof.

Now, let us prove the claim above.

According to the Postnikov resolution of $\pi$ as in the proof of Theorem \ref{thm:main3} $(C)$,
we have the following commutative diagram
\begin{equation*}
\begin{split}
\xymatrix{
& & [M, BU(3)] \ar[d]^{h_{\ast}} \ar[dr]^-{\pi_{\ast}}& & \\
[M, U] \ar[r]^-{\Omega c_{4}} & [M, K(\Z, 7)] \ar[r]^-{j_{\ast}} & [M, E] \ar[r]^-{q_{\ast}} & [M, BU] \ar[r]^-{c_{4}} & [M, K(\Z, 8)], 
}
\end{split}
\end{equation*}
where the bottom sequence is exact.
By the construction of the Postnikov resolution, the induced homomorphisms
\begin{align*}
h_{\ast} \colon \pi_{r}(BU(3)) \rightarrow \pi_{r}(E), \quad r \le 8,
\end{align*}
are isomorphic.
Hence, it follows from \cite[Lemma 4.1]{jt65} that 
\begin{align*}
h_{\ast} \colon [M, BU(3)] \to [M, E]
\end{align*}
is a bijection.  
Therefore, for any $\eta \in [M, BU]$ with $\eta \in \mathrm{Im} ~ \pi_{\ast}$, i.e., $c_{4}(\eta) = 0$, there is a one to one correspondence between 
\begin{align*}
\pi_{\ast}^{-1} (\eta) \quad \text{and} \quad q_{\ast}^{-1}( \eta ).
\end{align*}
Moreover, since $q$ is a principal fibration with fiber $K(\Z, 7)$ and $c_{4}$ as the classifying map,
it follows from the Lemma \ref{lem:delta} that there is a bijection between
\begin{align*}
q_{\ast}^{-1}( \eta ) \quad \text{and} \quad \mathrm{Cok} ~ \Delta (c_{4}, \eta).
\end{align*}
Hence it remains to prove that $\mathrm{Cok} ~ \Delta (c_{4}, \eta) = \mathfrak{T}$.

Let $\mu \colon BU \times BU \to BU$ be the canonical $H$-structure of $BU$.
It is known that
\begin{align*}
\mu^{\ast} (c_{4}) = c_{4} \times 1 + 1 \times c_{4} + c_{1} \times c_{3} + c_{2} \times c_{2} + c_{3} \times c_{1}.
\end{align*} 
Therefore, recall from Section \ref{s:pre} that
\begin{align*}
\Delta(c_{4}, \eta) \colon [M, U] \rightarrow [M, K(\Z, 7)]
\end{align*}
is a homomorphism given by
\begin{align*}
\Delta(c_{4}, \eta) ( f )  = & f^{\ast}\sigma (c_{4}) + ( f^{\ast} \sigma (c_{1}) ) ( \eta^{\ast}(c_{3})) + ( f^{\ast} \sigma (c_{2}) ) ( \eta^{\ast}(c_{2}))  + ( f^{\ast} \sigma (c_{3}) ) ( \eta^{\ast}(c_{1})) \\
= & f^{\ast} ( \gamma_{7} ) + f^{\ast} ( \gamma_{1} ) c_{3}(\eta) + f^{\ast} ( \gamma_{3} ) c_{2}(\eta) + f^{\ast} ( \gamma_{5} ) c_{1}(\eta),
\end{align*} 
for any $f \in [M, U]$.
It follows that
\begin{align*}
\mathrm{Cok} ~ \Delta(c_{4}, \eta) = \mathfrak{T} : = \frac{ H^{7}(M; \Z) } { \{ f^{\ast}( \gamma_7 ) + u_1 f^{\ast} ( \gamma_5 ) + u_2 f^{\ast} ( \gamma_3 ) + u_{3} f^{\ast} (\gamma_{1}) ~ | ~ f \in [M, U] \} }.
\end{align*}
The claim is proved.
\end{proof}



%


\end{document}